\documentclass[12pt]{amsart}
\usepackage{amscd}
\usepackage{verbatim}
\usepackage{amssymb, amsmath, amsthm, amscd,ifthen}
\usepackage[dvips]{graphics}
\usepackage[cp866]{inputenc}
\usepackage{graphicx}
\usepackage{epsfig}

\usepackage{amsmath,amssymb,amscd,}
\usepackage[all]{xy}

\usepackage{amssymb}

\newtheorem{thm}{Theorem}[section]

\newtheorem{lemma}[thm]{Lemma}

\theoremstyle{definition}

\newtheorem{dfn}[thm]{Definition}

\begin{document}
\author{A. Zhukova }

\title{Discrete Morse theory for the barycentric subdivision}

	\maketitle \setcounter{section}{0}

	\begin{abstract}
We work with discrete Morse theory. Let $F$ be a discrete Morse function on a simplicial complex $L$. We construct a discrete Morse function $\Delta(F)$ on the barycentric subdivision $\Delta(L)$. The constructed function $\Delta(F)$  "behaves the same way" as $F$, i. e. has the same number of critical simplexes and the same gradient path structure. 
	\end{abstract}

\section{Introduction}

Discrete Morse theory is a discrete analogue of the classical Morse theory. It was developed by R. Forman \cite{for}. This theory can be applied to any simplicial and regular CW-complexes and, although its definition is quite simple, many classical results analogous to the ones of the continuous Morse theory arise in its scope.

We develop a way to "transfer" a discrete Morse function, defined on a simplicial complex, onto the barycentric subdivision of this complex in such a way that all important data about this function (i. e. the number and dimensions of the critical simplexes and the gradient path structure) stays unchanged. It can be done in several different ways and we can produce several different Morse functions. It can be useful in computing multiplication in cohomology ring of this simplicial complex \cite{for}. The main result of the paper is the following theorem.

\begin{thm}\label{main_thm}
	Let $F$ be a discrete Morse function on a simplicial complex $L$. Suppose that for every critical cimplex $\alpha \in Crit(F)$ the ordering $Ord_{\alpha}$ on its vertices is chosen. Then the pairing on the barycentric subdivision $\Delta(L)$ constructed in Section \ref{sec_pairings} defines a discrete Morse function $\Delta(F)$ on $\Delta(L)$ such that the following holds:
	
	\begin{enumerate}

		\item  The critical simplexes of $\Delta(F)$ are exactly those that have the chosen orderings $Ord_*$  as their labels. That is, every critical $k$-simplex of $F$ contains exactly one critical $k$-simplex of $\Delta(F)$ which can be chosen arbitrary before constructing $\Delta(F)$. This defines a bijection  $Crit(F) \rightarrow Crit(\Delta(F))$

		\item There exists a natural bijection $Gr(F) \rightarrow Gr(\Delta(F))$ that respects the bijection $Crit(F) \rightarrow Crit(\Delta(F))$ defined above.
		
	\end{enumerate}
\end{thm}

It is worthy to mention in this respect that E.Babson and P. Hersch \cite{her}  introduced a technique that can be (as one particular application) used to build a certain Morse function on the barycentric subdivision of an arbitrary simplicial complex. This Morse function arises from a lexicographic order on the maximal chains of the simplexes of $L$, i.e. on the maximal simplexes of $\Delta(L)$. The question, whether there are connections between our work and \cite{her}, stays open.

The structure of this paper is as follows. In section \ref{sec_prelim} we give definitions of discrete Morse function and barycentric subdivision. We formulate our main theorem in this section as well. In section \ref{sec_pairings} we construct the Morse function on the barycentric subdivision of a simplicial complex and in section \ref{sec_proof} we prove that the constructed function satisfies the needed conditions.

\bigskip

\section{Preliminaries}\label{sec_prelim}

We start with the remindings.

\subsection*{Discrete Morse function on a regular simplicial complex \cite{for}}
Let $L$ be a regular simplicial complex. By $\alpha^p, \ \beta^p$ in this section we
denote  its $p$-dimensional simplexes, or \textit{$p$-simplexes}, for short.

 A \textit{discrete vector field} is a set of pairs
$$\big(\alpha^p,\beta^{p+1}\big)$$
 such that:
\begin{enumerate}
    \item  each simplex of the complex participates in at most one
    pair, and
    \item  in each pair, the simplex $\alpha^p$ is a facet of $\beta^{p+1}$.

\end{enumerate}

Given a discrete vector field, a \textit{path} of dimension $(p+1)$, or a \textit{$(p+1)$-path} is a sequence of
simplexes

$$ \beta_0^{p+1},\ \alpha_1^p,\ \beta_1^{p+1}, \ \alpha_2^p,\ \beta_2^{p+1} ,..., \alpha_m^p,\ \beta_m^{p+1},\ \alpha_{m+1}^p,$$
which satisfies the conditions:
\begin{enumerate}
    \item  Each $\big(\alpha_i^p,\ \beta_i^{p+1}\big)$ is a pair.
    \item Whenever $\alpha$ and  $\beta$ are neighbors in the path,
    $\alpha$ is a facet of $\beta$.
    \item $\alpha_i\neq \alpha_{i+1}$.
\end{enumerate}

Every path consists of "face"-steps (that is, transition from simplex $\beta_i^{p+1}$ to one of its faces $\alpha_{i+1}^p$) and "pair"-steps (that is, transition from simplex $\alpha_{i}^p$ to simplex $\beta_i^{p+1}$, following the pairing).

A path is \textit{closed} if $\alpha_{m+1}^p=\alpha_{0}^p$.

A \textit{discrete Morse function on a regular simplicial complex} is a
discrete vector field with no closed paths in it.

Assuming that a discrete Morse function is fixed, the \textit{critical simplexes} are those simplexes of the complex that
  are not paired.  We denote the set of all critical simplexes of a discrete Morse function $F$ with $Crit(F)$.

A \textit{ gradient $(p+1)$-path} of a discrete Morse function leading from
one critical simplex $\beta^{p+1}$ to some another  critical simplex
$\alpha^{p}$ is the sequence of simplexes:
$$\beta^{p+1},\ \ \alpha_1^p,\ \beta_1^{p+1},\ \alpha_2^p,\ \beta_2^{p+1},\ ...,\ \alpha_m^p,\ \beta_m^{p+1},\ \alpha^{p}$$

satisfying the three above conditions.

We denote the set of all gradient paths of a discrete Morse function $F$ with $Gr(F)$.

\subsection{Barycentric subdivision}

\begin{dfn}
	Let $L$ be a simplicial complex. Then its barycentric subdivision $\Delta(L)$ is a simplicial complex, such that the vertices of $\Delta(L)$ are in a bijectional correspondence with the ste of all simplexes of $L$ and the subset of vertices in $\Delta(L)$ forms a simplex iff the corresponding simplexes form a chain in the poset of all simplexes of $L$.
\end{dfn}

For any two simplexes $\alpha, \beta \in \Delta(L)$ we have $\alpha \in \beta$ iff the chain that corresponds to $\alpha$ is a subchain of the chain that corresponds to $\beta$.

The regular realization of $\Delta (L)$ can be constructed for any regular realisation of $L$ as follows. We realize every vertex of $\Delta(L)$ as the barycenter of the realization of the corresponding simplex of $L$, as depicted on Fig. \ref{fig_sub}. The realisation of a simplex $\alpha \in \Delta(L)$ lies in the interior of the realisation of a simplex $\beta \in L$ iff $\beta$ is the last simplex in the chain that corresponds to $\alpha$.

\begin{figure}[h]\label{fig_sub}
\centering
\includegraphics[width=10 cm]{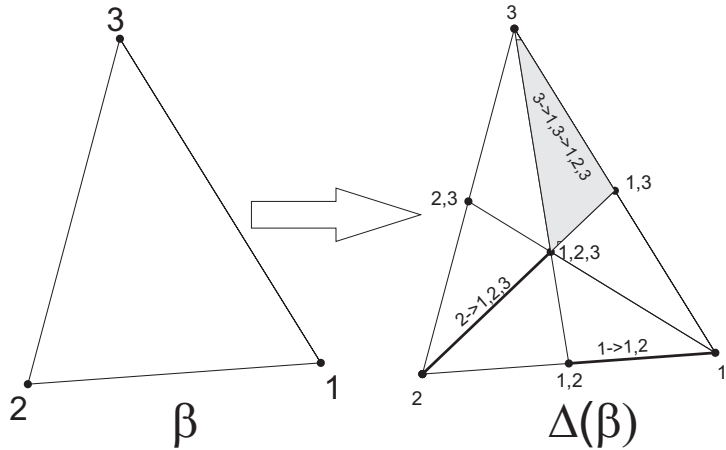}
\caption{Barycentric subdivision of a 2-simplex with the chains corresponding to some simplexes}
\end{figure}

To simplify our construction, we transform the chains of the simplexes of $L$ in the following way. We turn a chain
 $$ s_1\rightarrow s_2\rightarrow\dots\rightarrow s_{k+1}$$ that corresponds to the simplex $\gamma$ into the ordered set 
 $$\lambda(\gamma)=\{s_1, s_2 \backslash s_1, \dots, s_{k+1}  \backslash s_k \},$$ that we call the \textit{label} of $\gamma$. It is a linearly ordered partition of the set $s_{k+1}$. 
 
In this notation a simplex $\alpha' \in \Delta(L)$ is a face of a simplex $\alpha \in \Delta(L)$ iff the partition $\lambda(\alpha)$ can be turned into a refinement of the partition $\lambda(\alpha')$ by deleting some sets from the end of $\lambda(\alpha)$. We will use this geometrical picture in our proofs.

\medskip

\textbf{Example} The faces of the triangle in $\Delta(L)$ with the label $$(\{a,f\} \ \{d\} \ \{t\})$$ are the following:
\begin{itemize}
\item 1-dimensional (edges): $(\{a,f\} \{d,t\}), (\{a,f,d\} \ \{t\}), (\{a,f\} \ \{d\} )$.

\item 0-dimensional (vertices): $(\{a,f,d,t\}), (\{a,f,d\} ), (\{a,f\})$.
\end{itemize}
\medskip

\section{Pairing on the complex $\Delta(L)$}\label{sec_pairings}

Suppose we have a discrete Morse function $F$ on a regular simlicial complex $L$. In this section we define  a discrete vector field on the complex $\Delta(L)$. We deal with pieces of critical and non-critical simplexes of $L$ in two different ways.

\subsection{Non-critical simplexes}
	Let $(\alpha^{n-1},\beta^n) \in F$. We renumber their vertices so that $\alpha=[n]=\{1,2, \dots,n\},\beta=[n+1]=\{1,2, \dots, n+1\}$. This renumbering is almost arbitrary, except the vertex $n+1 = \beta \backslash \alpha$. The pairing that we define below do not depend on this renumbering, so it is made only for conveniency.
	
	 We  define pairings on all simplexes of $\Delta(L)$, that lie in the interiors of $\alpha$ and $\beta$. They are exactly all simplexes whose labels are the subdivisions of $[n]$ and $[n+1]$. Let $\gamma$ be a $k$-simplex with such label.	 Consider four possible cases:

	\begin{enumerate}
	\item   $\lambda(\gamma)$ is an oredered subdivision of $[n]$, i. e. $\gamma \in \Delta(\alpha)$. Then we obtain a pair for $\gamma$ by adding a singleton $\{n+1\}$ to the right end of $\lambda(\gamma)$. We get a $(k+1)$-simplex that belongs to the case 2.
	
	\item The entry $n+1$ forms a singleton  in $\lambda(\gamma)$ and it is the last set in $\lambda(\gamma)$. Then we obtain a pair for $\gamma$ by deleting $\{n+1\}$ from $\lambda(\gamma)$. We get a $(k-1)$-simplex  that belongs to the case 1.
				
	\item The entry $n+1$ forms a singleton  in $\lambda(\gamma)$ and it is not the last set in $\lambda(\gamma)$ Then we obtain a pair for $\gamma$ by uniting the singleton $\{n+1\}$ with the set that goes after it. We get a $(k-1)$-simplex  that belongs to the case 4.
	
	\item The entry $n+1$ lies in a non-singleton set in $\lambda(\gamma)$.  Then we obtain a pair for $\gamma$ by  splitting the  entry $n+1$ to the left from the set containing it and forming a singleton $\{n+1\}$. We get a $(k+1)$-simplex  that belongs to the case 3.

	\end{enumerate}  
	
	It is easy to see that every simplex takes part in exactly one pair. 
	
\textbf{Example} If $n+1=5$, then we will have such pairs of simplexes as the two below:
	
	$$((\{1\} \ \{3,4\} \  \{2, 5\}),\  (\{1\} \ \{3,4\} \  \{5\} \ \{2\})),$$
	
	and
	
	$$(\{1\} \ \{3,4\} \  \{2\}),\  (\{1\} \ \{3,4\} \  \{2\} \  \{ 5\}))).$$

\medskip

Now prove that this pairing defines a Morse function inside $\Delta(\beta)$.
\begin{lemma}\label{good_pair}
There are no cyclic paths in the vector field on $\Delta(\beta)$ defined above. 
\end{lemma}

\begin{proof}

Suppose we have a path $\Gamma$ in the pairing defined above. Consider the position of the entry $n+1$ in the labels of the simplexes in $\Gamma$. In every "pair"-step this entry exits some non-singleton set and forms a singleton. Therefore, every "face"-step, except maybe the first or the last step in the path is performed by adding this entry to the set next to it. If we add this entry to the right set, then during the next "pair"-step we will immediately return back to this simplex, which is forbidden. So every face-step is defined uniquely and the entry $n+1$ travels to the left side of the label during the path. Therefore, no path is cyclic.  
\end{proof}

As we will see in the Lemma \ref{lemma_proper}, we are interested only in $n$-paths inside $\Delta(\beta)$. The lemma below follows from the pairing construction and the definition of the barycentric subdivision.

	\begin{lemma}\label{lemma_noncritical}
		Let $\lambda(\gamma) = (I_1 \ I_2 \dots I_{n+1})$ be an $n$-simplex in $\Delta(\beta)$ (note that in this case all the sets $I_i$ are singletones). Then the following statements hold.
		\begin{enumerate}
		\item There is exactly one $(n-1)$-face of $\gamma$ that lies on the boundary of $\beta$ and it is the simplex with the label $( I_1 \ I_2 \dots  I_{n})$;
		
		\item The simplex $\gamma$ is paired with the $(n-1)$-simplex  given above  iff $I_{n+1}=\{n+1\}$;

		\item If $I_1 \neq \{n+1\}$, then there is exactly one $(n-1)$-face of $\gamma$, that is paired with another $n$-simplex of $\Delta(\beta)$.
		 It has the label obtained by uniting in $\lambda(\gamma)$ the singleton $\{n+1\}$ with the singleton that goes before it;
		 This $(n-1)$-face of  $\gamma$ is paired with the simplex $\gamma'$, which label can be obtained from $\lambda(\gamma)$ by interchanging the singleton $\{n+1\}$ with the singleton that goes before it.
		  
		\item If $I_1=\{n+1\}$, then there are no $(n-1)$-faces of $\gamma$ that are paired with another k-simplexes of $\Delta(\beta)$.
		
		\end{enumerate}
	\end{lemma}
	
	This lemma shows that if we construct an $n$-path that goes through $\Delta(\beta)$ we do not have much choice. We can start from any $(n-1)$-simplex  in $\Delta(\alpha)$ and follow the pairing. On every "face" step we can either go to the boundary of $\Delta(\beta)$ in the uniquely defined way ( Lemma \ref{lemma_noncritical}.1) or go further in the uniquely defined way (Lemma \ref{lemma_noncritical}.3), until we get to the simplexes that lie near the vertex $n+1$ (Lemma \ref{lemma_noncritical}.4).
	
		Informally, these paths form a "flow" from $\alpha$ in the direction of the vertex $p+1$ (see Fig. \ref{fig_rairing} for example).

\label{fig_rairing}\begin{figure}[h]
\centering
\includegraphics[width=10 cm]{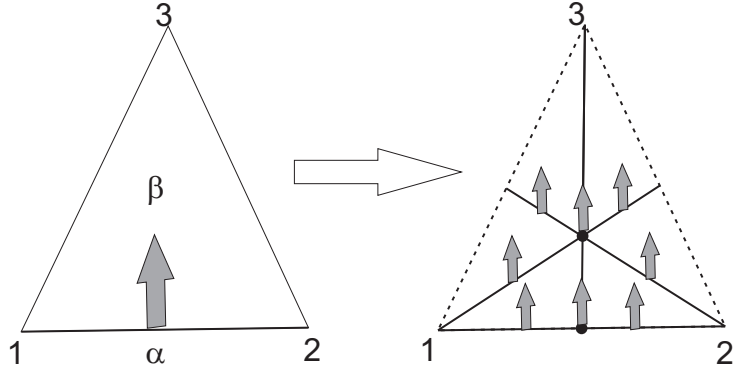}
\caption{The pairings on the barycentric subdivision of a pair $(\alpha^n, \beta^{n+1})$, for $n+1=3$}
\end{figure}

		 The next lemma follows from the above.

\begin{lemma}\label{path_pair}
Suppose that a gradient $n$-path of $\Delta(F)$ goes through $\Delta(\beta)$ and the last simplex of $\Delta(\beta)$ in this path is the $(n-1)$-simplex $\lambda(\gamma)=(I_1 \  I_2 \dots  I_{n})$. Suppose $I_j=\{n+1\}$.

Then  the first simplex of this path that lies in $\Delta(\beta)$ has the label

$$(I_1  \ I_2 \dots I_{j-1} \  I_{j+1}  \ \dots \ I_{n} \ [n+1] \backslash \bigcup_{i=1}^{n} I_i ).$$
\end{lemma}

\textbf{Example} 

Let $n+1=5$. If an n-path  goes through $\Delta(\beta)$ and the last simplex of $\Delta(\beta)$ in this path is the $(n-1)$-simplex labeled $(\{1\} \ \{3\} \ \{5\} \ \{4\} )$, then the first simplex of $\Delta(\beta)$ in this path is the simplex labeled $(\{1\} \ \{3\} \ \{4\} \ \{2\})$ and the path looks the following way:

 $$(\{1\} \ \{3\} \ \{4\} \ \{2\})$$
 $$(\{1\} \ \{3\} \ \{4\} \ \{2\} \ \{5\})$$
  $$(\{1\} \ \{3\} \ \{4\} \ \{2,5\})$$  
   $$(\{1\} \ \{3\} \ \{4\} \ \{5\} \ \{2\})$$
      $$(\{1\} \ \{3\} \ \{4, 5\} \ \{2\})$$
  $$(\{1\} \ \{3\} \ \{5\} \ \{4\} \ \{2\})$$
    $$(\{1\} \ \{3\} \ \{5\} \ \{4\}).$$
  
\bigskip

\subsection{Critical simplexes} 

Let $\alpha^n$ be a critical $n$-simplex of $F$. We renuber its vertices so that $\alpha=[n+1]=\{1,2, \dots , n+1\}$.  We will pair the simplexes of $\Delta(L)$ that lie in the inner part of $\alpha$, i. e. those which labels are the ordered subdivisions of the set $[n+1]$. We leave only one $n$-simplex non-paired, namely the simplex with the label 

 $$\lambda(\alpha')=(\{n+1\} \ \{n\} \dots \{1\} ).$$

 This simplex depends on our renumbering, and, given an arbitrary ordering $Ord_{\alpha}$, we can renumber the vertices in $\alpha$ in the order opposite to $Ord_{\alpha}$ to get $\lambda(\alpha')= Ord_{\alpha}$.

\medskip
Let $\gamma$ be a $k$-simplex in $\Delta(\alpha)$ with the label $\lambda(\gamma)= (I_1 \ I_1 \dots I_{k+1})$. Let $i$ be the length of the greatest common suffix of $\lambda(\gamma)$ and $\lambda(\alpha')$ (i. e. for every $j\leq i$ the j-th from the end sets in $\lambda(\gamma)$ and $\lambda(\alpha')$ are equal). Three cases are possible:

\begin{enumerate}
\item $i=n$ Then $\gamma=\alpha'$ and we do not pair it;
\item $i<n$ and the entry $i+1$ makes a singleton in $\lambda(\gamma)$. Then we pair the simplex $\gamma$ with a $(k-1)$-simplex, which label is obtained from $\lambda(\gamma)$ by uniting this singleton with the set that goes after it. This does not change the greatest common suffix with $\lambda(\alpha')$.
\item $i<n$ and the entry $i+1$ lies in a non-singleton set in $\lambda(\gamma)$.  Then we pair the simplex $\gamma$ with a $(k+1)$-simplex, which label is obtained from $\lambda(\gamma)$ by splitting this entry from this set to the left of it and forming a singletone. This does not change the greatest common suffix with $\lambda(\alpha')$
\end{enumerate}

We always pair one simplex of type $2$ with one simplex of type $3$, so this pairing  is well-defined.

      \textbf{Example.} For $n=5$ we build such pairs as
       $$((\{1,2\} \ \{3\} \ \{4,5\}), \ (\{1\} \ \{2 \} \ \{3\} \ \{4,5\})),$$ 
       where the length of the greatest common suffix with $\lambda(\alpha')$ is $0$, and 
      $$((\{4,5\} \ \{3\} \ \{2\} \ \{1\}),  \ (\{4\} \ \{5\} \ \{3\} \ \{2\} \ \{1\})),$$
where the length of the greatest common suffix with $\lambda(\alpha')$ is $3$.
\medskip 
Now we prove that this pairing forms a Morse function on $\Delta(\alpha)$.

\begin{lemma}\label{good_nonpair}
There are no cyclic paths in the vector field on $\Delta(\alpha)$ defined above.
\end{lemma}
\begin{proof}
Let $\Gamma$ be a path in this field. Consider the greatest common suffix of the labels of simplexes in $\Gamma$ and $\lambda(\alpha')$. It never grows during the path. If it is getting shorter, then the path cannot be cyclic.

 Suppose this suffix stays the same during $\Gamma$ and its length is $i$. Then all the "pair"-steps in $\Gamma$ are performed by splitting the entry $i+1$ from non-singleton set to the left. Therefore, all the "face"-steps in  $\Gamma$ except, maybe, the first and the last one are performed by adding this entry to the set on the left side. So, the entry $i+1$ travels to the left in the label and $\Gamma$ cannot be cyclic.

\end{proof}

\begin{lemma}\label{path_nonpair}
For every $(n-1)$-simplex $\gamma \in \Delta(\alpha)$ that lies  on the boundary of $\alpha)$ there is a unique $n$-path in $\Delta(\alpha)$ that starts in $\alpha'$ and exits $\Delta(\alpha)$ through $\gamma$.
\end{lemma}	 

\begin{proof} As we already know, there is a unique $n$-simplex $\gamma' \in \Delta(\alpha)$ that has $\gamma$ on its boundary. Our path must go through $\gamma'$.

The greatest common suffix with $\lambda(\alpha')$ can only decrease during the path (at the start its length is $n+1$). For arbitrary $i$, the entry $i$ moves inside the label during the path only when this suffix has length $i-1$. So by the time the length of this suffix becomes smaller than $i-1$, the permutation of the entries $i, i+1, \dots n+1$ in the label is fixed and does not change any more. If such a path exists, then all entries take their places in the desreasing order, and this imply the uniqueness of the path. Knowing that, it is not hard to construct such a path. For example, if $n=4$ and $\gamma=(\{2\} \ \{4\} \ \{5\} \ \{1\})$, then $\gamma'=(\{2\} \ \{4\} \ \{5\} \ \{1\} \ \{3\})$ and the path looks as follows:

$$\alpha'=(\{5\} \ \{4\} \ \{3\} \ \{2\} \ \{1\})$$
$$(\{4, 5\} \ \{3\} \ \{2\} \ \{1\})$$
$$(\{4\} \ \{ 5\} \ \{3\} \ \{2\} \ \{1\})$$
$$(\{4\} \ \{ 5\} \ \{2,3\} \ \ \{1\})$$
$$(\{4\} \ \{ 5\} \ \{2\} \ \{3\} \ \{1\})$$
$$(\{4\} \ \{ 2,5\} \ \{3\} \ \{1\})$$
$$(\{4\} \ \{ 2\} \ \{5\} \ \{3\} \ \{1\})$$
$$(\{2,4\} \ \{5\} \ \{3\} \ \{1\})$$
$$(\{2\} \ \{ 4\} \ \{5\} \ \{3\} \ \{1\})$$
$$(\{2\} \ \{ 4\} \ \{5\} \ \{1,3\})$$
$$\gamma'=(\{2\} \ \{4\} \ \{5\} \ \{1\} \ \{3\})$$
$$\gamma=(\{2\} \ \{4\} \ \{5\} \ \{1\})$$
\end{proof}

\begin{figure}[h]\label{fig_noncritical}
\centering
\includegraphics[width=10 cm]{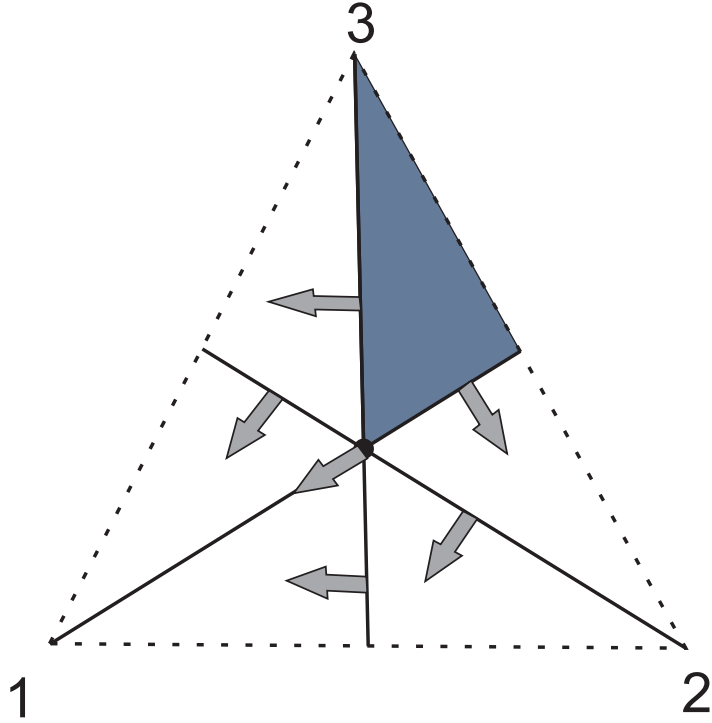}
\caption{The pairings on the barycentric subdivision of a simplex $\alpha$ for $n=2$}
\end{figure}

Any gradient path that contains $n$-simplexes from $\Delta(\alpha)$ can look one of the two 
 following ways.
 \begin{enumerate}
 \item It starts from $\alpha'$ and make steps inside $\Delta(\alpha)$. Then at some "face"-step this path comes out of $\Delta(\alpha)$ through a $(n-1)$-simplex on the boundary. According to Lemma \ref{path_nonpair}, this path is defined by its exit simplex uniquely.

\item It comes into $\alpha'$ at some "pair"-step from some $(n+1)$-simplex $\gamma \in \Delta(L)$. From the structure of the barycentric subdivision follows, that simplex $\gamma$ lies in the barycentric subdivision of some $(n+1)$-simplex $\beta$ of $L$, and the simplex $\alpha$ is a facet of $\beta$. This path is an $(n+1)$-path and it ends at $\alpha'$.

\end{enumerate}

\section{Paths on the barycentric subdivision}\label{sec_proof}

The Theorem \ref{main_thm}.1 follows from the construction of $\Delta(F)$. In this section we prove the rest of  the Theorem \ref{main_thm}. 

\begin{lemma}\label{lemma_cyclic}
Let $F$ be a discrete Morse function on a simplicial complex $F$. Then the pairing $\Delta(F)$ constructed in the section \ref{sec_pairings} is a Morse function
\end{lemma}
\begin{proof}
We need to prove that no path in $\Delta(F)$ is cyclic. By Lemmata \ref{good_pair} and \ref{good_nonpair}, it is true if the path stays inside one simplex of $L$. 

Let $\Gamma$ be a path in $\Delta(F)$ and let $\Gamma$ include parts of more than one simplex of $L$. Every simplex $\gamma \in \Gamma$ lies in the inner part of some simplex in $L$.
Take these simplexes as a sequence and delete the repeats. The resulting sequence $\Gamma'$ is cyclic if $\Gamma$ is cyclic.

From the definition of $\Delta(F)$ follows, that in every two consequent simplexes in $\Gamma'$ one is a facet of another. Moreover, if the  simplex with the lower dimension goes before the simplex with the higher dimension, then these two simplexes are paired. So, the sequence $\Gamma'$ consists of "face"-steps and "pair"-steps. By the definition of Morse function, no two "pair"-steps can go consequently.

 If no  two "face"-steps go consequently in $\Gamma'$, then $\Gamma'$ is a path in $F$ and cannot be cyclic. If there are at least  two consequent "face"-steps, then the dimension of the simplexes decreases during $\Gamma'$ more times than increases and $\Gamma'$ cannot be cyclic.
\end{proof}

Now we consider for arbitrary $n$ the way that the gradient $n$-paths behave in $\Delta(F)$.
If a gradient $n$-path  starts in a critical simplex $\alpha'$ then it leaves the corresponding simplex $\alpha$ of $L$ through a $(n-1)$-simplex that lies in $(n-1)$-face of $\alpha$. If an $n$-path enters $\Delta(\beta)$, where $\beta$ is non-critical through an $(n-1)$ simplex on its boundary, then it leaves this simplex through another $(n-1)$-simplex that as well lies on the boundary of $\beta$. So this path never gets out of $n$-simplexes of $L$. We get the following lemma.
\begin{lemma}\label{lemma_proper}
Suppose $\Gamma$ is a gradient $n$-path. Then all $n$-simplexes in $\Gamma$ lie in the interiors of the $n$-simplexes of $L$.
\end{lemma}

 In other words, a gradient $n$-path in $\Delta(F)$ never gets inside the simplexes of $L$ of dimensions  higher than $n$.

Now prove that the critical path structure of the function $\Delta(F)$  is isomorphic to the gradient path structure of the function $F$. We do it by constructing two mappings between the set of gradient paths $Gr(F)$  of the function $F$ and the set of gradient paths $Gr(\Delta(F))$ of the function $\Delta(F)$. These mappings are one-to-one, they are opposite to each other and they respect our bijection $Crit(F) \rightarrow Crit(\Delta(F))$.

\medskip

\textbf{Map $Gr(\Delta(F)) \rightarrow Gr(F)$}

Let $\Gamma$ be a gradient $n$-path in $\Delta(F)$. We construct the corresponding sequence $\Gamma'$ of simplexes of $L$, like we did in the proof of the Lemma \ref{lemma_cyclic}. By Lemma \ref{lemma_proper}, and by construction of $\Delta(F)$ this sequence is a $n$-path in $F$. Moreover, it starts and ends in critical points, since only critical points of $F$ contain critical points of $\Delta(F)$. These critical simplexes correspond to the start and the end of $\Gamma$.

Therefore, $\Gamma'$ is a gradient path of $F$.

\medskip

\textbf{Map $Gr(\Delta(F)) \rightarrow Gr(F)$}

Let $\Gamma$ be a gradient $n$-path in $F$ from a simplex $\beta$ to a simplex $\alpha$:

$$\beta=\beta_0,\  \alpha_1,\ \beta_1,\ \alpha_2,\ \beta_2,\ ...,\ \alpha_k,\ \beta_k,\ \alpha$$

 We construct a corresponding path $\Delta(\Gamma)$ in $\Delta(F)$ that goes from $\beta'$ to $\alpha'$. 
 We define the path inside $\Delta(\beta_i)$ for every $i$ inductively. We start from the last $n$-simplex in $\Gamma$. Our path exits $\Delta(\beta_k)$ through $\alpha'$, which by Lemma \ref{good_pair} defines the path in $\Delta(\beta_k)$ uniquely. For every $1 \leq i \leq k$ the first simplex of the path constructed in $\Delta(\beta_i)$ becomes the last simplex of the path in $\Delta(\beta_{i-1})$ and defines the path in $\Delta(\beta_{i-1})$ uniquely. For the simplex $\beta$ the same holds by the Lemma \ref{good_nonpair}. Therefore, the path $\Delta(\Gamma)$ is defined uniquely.

\bigskip

It is easy to see that for every gradient path $\Gamma$ in $F$ we have $(\Delta(\Gamma))'=\Gamma$ and from Lemma \ref{lemma_proper} follows that for every gradient path $\Gamma$ in $\Delta(F)$ we have $(\Delta(\Gamma'))=\Gamma$. Therefore, our mappings define the bijection between the path structure on $F$ and $\Delta(F)$. Theorem \ref{main_thm} is proved.

\textit{Acknowledgements. This work is supported by  Russian Science Foundation 16-11-10039. The author is a Young Russian Mathematics award winner and would like to thank its sponsors and jury. The author thanks G. Panina for formulating the problem.
}

Zhukova Alena Mikhailovna\\
 Saint-Petersburg  State University\\
  Russia, Saint-Petersburg, Universitetskaya nab. 7-9\\ 
 a.zhukova@spbu.ru\\

\end{document}